\documentclass[journal,onecolumn,12pt]{IEEEtran}
\usepackage{amsmath}
\usepackage{amssymb}
\usepackage{amsthm}
\usepackage{multirow}
\usepackage{color}
\usepackage{longtable}
\usepackage{array}
\usepackage{url}
\usepackage{enumerate}
\usepackage{enumitem}
\usepackage{caption}
\usepackage{setspace}
\usepackage{graphicx}

\newtheorem{theorem}{Theorem}[section]

\newtheorem{lemma}[theorem]{Lemma}
\newtheorem{example}[theorem]{Example}

\newtheorem{remark}{Remark}[section]

\ifCLASSINFOpdf
\else
\fi
\begin{document}
%
\title{Placement Delivery Arrays from Combinations of Strong Edge Colorings}
%
%
%

\author{Jerod Michel and Qi Wang,~\IEEEmembership{Member,~IEEE}
\thanks{J. Michel is with the Department of Mathematics, Nanjing Univeristy of Aeronautics and Astronautics, Nanjing, Jiangsu, China, michel@nuaa.edu.cn.}\thanks{Q. Wang is with the Department of Computer Science and Engineering, Southern University of Science and Technology, Shenzhen, Guangdong, China, wangqi@sustech.edu.cn.}\thanks{The authors were supported in part by the National Natural Science Foundation of China under Grant No. 61672015.}}
\maketitle

\begin{abstract}
It has recently been pointed out in both of the works [C. Shanguan, Y. Zhang, and G. Ge, {\em IEEE Trans. Inform. Theory}, 64(8):5755-5766 (2018)] and [Q. Yan, X. Tang, Q. Chen, and M. Cheng, {\em IEEE Commun. Lett.}, 22(2):236-239 (2018)] that placement delivery arrays (PDAs), as coined in [Q. Yan, M. Cheng, X. Tang, and Q. Chen, {\em IEEE Trans. Inform. Theory},  63(9):5821-5833 (2017)], are equivalent to strong edge colorings of bipartite graphs. In this paper we consider various methods of combining two or more strong edge colorings of bipartite graphs to obtain new ones, and therefore new PDAs. Combining PDAs in certain ways also gives a framework for obtaining PDAs with more robust and flexible parameters. We investigate how the parameters of certain strong edge colorings change after being combined with others and, after comparing the parameters of the resulting PDAs with those of the initial ones, find that subpacketization levels thusly can often be improved.
\end{abstract}


%
\IEEEpeerreviewmaketitle

\section{Introduction}\label{sec1}
A dramatic increase in the demand for video delivery via wireless networks is now one of the main driving factors behind the study of centralized coded caching. The prevalent scenario can be described as a group of users, each of whom connects to a server with a substantial library of files and, at various points in time, demands specific files from the server. Excessive, simultaneous demands can lead to the jamming up of networks. Maddah-Ali and Niesen, in \cite{AN}, proposed the {\it centralized coded caching scheme} as a solution to this problem. The central idea is to design an appropriate content placement strategy where, in the delivery phase, the demands of users can be met with a relatively low number of multicast transmissions. Users should be able to use the contents broadcast in the delivery phase together with those stored in their local caches to recover the requested files.
\par
Assuming that there are $N$ (unit-size) files and $K$ users, each of which has a local cache of size $M$ (units), each file can be split up into $F$ packets, and the total transmission amount in the delivery phase, called the {\it rate}, is denoted by $R$. The main indicators when evaluating the performance of a coded caching scheme are $R$ and $F$, i.e., for a fixed ratio $M/N$, the behavior of $R$ and $F$ should be treated as functions of $K$. In most existing coded caching schemes, $F$ increases (usually exponentially) with the number of users $K$. This is less than practical when $K$ is large, thereby making it worthwhile to consider caching schemes that require a smaller rate of increase of $F$ as a function of $K$. There have been several recent works on coded caching schemes, see for example \cite{JI}, \cite{KAR}, \cite{NMA}, \cite{PED},\cite{SHANM}.
\par
Yan et al., in \cite{YAN0}, represent the coded caching scheme by a single array which they call a {\em placement delivery array} (PDA). In short, a PDA can be thought of as having two phases. The placement phase splits files into packets and places them into the local caches of individual users. The delivery phase shows each user what has been cached as well as what should be transmitted. The problem of designing a coded caching scheme thus becomes that of designing the appropriate PDA for some given parameters. Although the schemes proposed by Yan et al. in \cite{YAN0} have significantly lower subpacketization levels than that of the Ali-Niesen schemes (i.e. $F$ grows at a much slower rate with respect to $K$) \cite{AN} at the cost of a slight increase in rate, $F$ still increases exponentially with $K$. Several new methods for constructing PDAs have since been reported where the problem of reducing subpacketization levels is considered.
\par
Two methods for constructing PDAs are especially relevant to this paper. In \cite{CHONG}, Shanguan, Zhang and Ge discerned an important connection between PDAs and $3$-uniform, $3$-partite hypergraphs with the $(6,3)$-free property as well as construct PDAs for which $F$ increases subexponentially with $K$. This connection was further expanded upon in terms of strong edge colorings of bipartite graphs by Yan et al. in \cite{YAN1}, thereby bridging the study of strong edge colorings (and the strong chromatic index of graphs) and PDAs. There have been other recent works concerning the problem of reducing subpacketization levels as well. In \cite{SHANM0} and \cite{SHANM1} Shanmugam et al. constructed coded caching schemes with reduced subpacketization levels using Ruzsa-Szemeredi graphs, in \cite{TANG} Tang and Ramamoorthy discuss an approach using linear block codes, in \cite{KRISH} Krishnan discussed an approach using line graphs of bipartite graphs, and in \cite{AGR} Agrawal used combinatorial designs to construct coded caching schemes with reduced subpacketization. There have also been some attempts at approaching this problem using other strategies, e.g., see \cite{LAMP} and \cite{OZ}.
\par
The present paper investigates the prospect of constructing new PDAs by taking different types of products of their underlying strong-edge colored bipartite graphs. We not only obtain PDAs with new parameters, but we also show how it is possible to reduce subpacketization levels by such methods. By examining the effects (and the trade-offs therein) on the parameters of existing PDAs after their underlying strong edge colorings have been combined in certain ways with others, we give an initial framework for improving subpacketization levels.
\par
The remainder of this paper is organized as follows. In Section \ref{sec2} we briefly review coded caching shemes, the PDA design introduced in \cite{YAN0} as well as its connection with strong edge colorings of bipartite graphs which was shown in \cite{YAN1} (and indirectly in \cite{CHONG}). In Section \ref{sec3} we discuss combining strong edge colorings in certain ways to obtain new strong edge colorings. In Section \ref{sec4} we look at the construction of PDAs through strong edge colorings of tensor-like products of certain types of graphs. Section \ref{sec5} discusses the construction of PDAs through strong edge colorings of strong-like products of certain types of graphs. In Section \ref{sec6} we compare the parameters of some PDAs to those of the resulting PDA after the former has been combined with another in one of the ways discussed in this paper, and Section \ref{sec7} concludes the paper.
\section{Preliminaries}\label{sec2}
Here we briefly review the necessary concepts related to the centralized coded caching scheme model, PDAs and strong edge colorings of bipartite graphs.
\subsection{Centralized Coded Caching Scheme}
Here we recall the centralized coded caching scheme introduced in \cite{AN}. A caching system with a single server is connected to $K$ users, denoted by $\mathcal{K}=\{1,...,K\}$, via an error-free shared link. In the server are stored $N$ ($N\geq K$) files, denoted by $\mathcal{W}=\{W_{1},...,W_{N}\}$, each of which is assumed to be of unit size. Each user has a cache of size $M$ units ($0\leq M\leq N$). The illustration of such a scheme, which can also be found in many previous papers, is given in Fig. \ref{fig0}.
\begin{figure}[!htb]
\begin{center}
\includegraphics[scale=0.50]{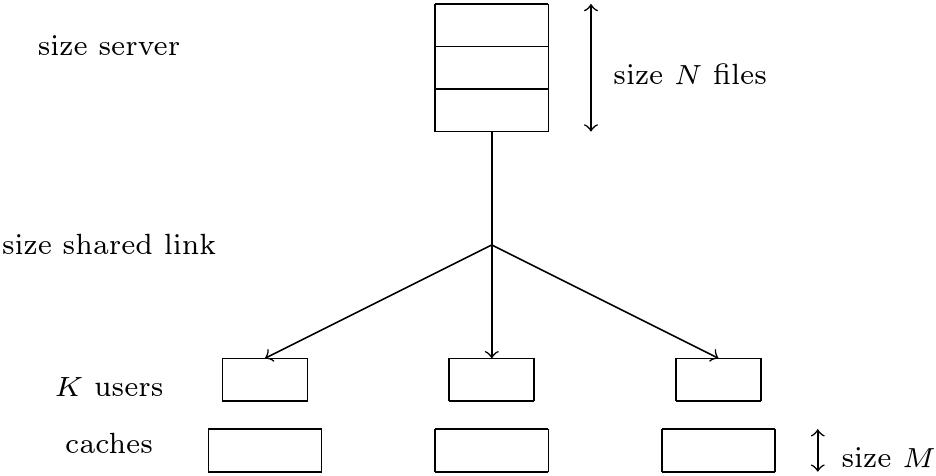}
\caption{Caching system setting.}
\label{fig0}
\end{center}
\end{figure}
As was discussed in \cite{AN} and \cite{YAN0}, the caching system operates in two separated phases:
\begin{enumerate}
\item[(1)] {\bf Placement Phase} In this phase, a file is subdivided into $F$ packets of equal size $1/F$ units. Then each packet is placed deterministically into each of the different users' caches. The total space taken up by the packets at each user cannot exceed its cache size $M$ (so that the number of packets is at most $\lfloor MF\rfloor$).
\item[(2)] {\bf Delivery Phase} In this phase each user may independently request a file from $\mathcal{W}$ at random. The request is denoted by ${\bf d} = (d_{1},...,d_{K})$, where user $k$ requests the file $W_{d_{k}}$, with $k\in \mathcal{K}$ and $d_{k}\in\{1,...,N\}$. On receiving the request ${\bf d}$, the server broadcasts an XOR multiplexing of packets. Each user is able to recover its requested file from the signal received in the delivery phase together with the content of its own cache.
\end{enumerate} We will often refer to $F$ as the {\em subpacketization level}.
\subsection{Placement Delivery Arrays} Again let $\mathcal{W}=\{W_{1},...,W_{N}\}$ be the set of files, $\mathcal{K}=\{1,...,K\}$ denote the $K$ users, and $\mathcal{F}=\{1,...,F\}$ be a set of cardinality $F$. 
Yan et al., in \cite{YAN0}, proposed the PDA as a method for designing a $(K,M,N)$ coded caching scheme that can be expressed in a single array. A {\em PDA} is an array $\mathcal{P}=\left[p_{j,k}\right]_{F\times K}$ of size $个F\times K$, where both $F$ and $FM/N$ are integers. The entries of the array are either the special symbol `` * '' or come from the set of integers $\mathcal{S}=\{1,...,S\}$, and each integer $s\in \mathcal{S}$ must appear at least once in the array. The following constraints must also be satisfied:
\begin{itemize}[leftmargin=*]
\item[(A)] The symbol `` * '' must appear $Z=FM/N$ times per column. (Therefore each column has $F-Z$ integer entries.)
\item[(B)] No integer can appear more than once per row or per column.
\item[(C)] For any two distinct entries $p_{j_{1},k_{1}}=p_{j_{2},k_{2}}=s\in \mathcal{S}$ where $j_{1}\neq j_{2}$ and $k_{1}\neq k_{2}$, we have that $p_{j_{1},k_{2}}=p_{j_{2},k_{1}}=*$.
\end{itemize}
Such an array satisfying the above constraints is referred to as a $(K,F,Z,S)$-{\em PDA}. 
Given a $(K,F,Z,S)$-PDA, a caching scheme may be implemented as follows.
\begin{itemize}
\item[(1)] In the placement phase a file is sub-divided into $F$ packets of equal size $\frac{1}{F}$, i.e., $W_{i}=\{W_{i,j}\mid j\in\mathcal{F}\}$. Each user has access to the set of files $\mathcal{W}$. The user $k\in \mathcal{K}$ receives the following packets in their cache: $Z_{k}=\{W_{i,j}\mid p_{j,k}=*, i=1,...,N\}$.
\item[(2)] In the delivery phase each user independently requests one file from $\mathcal{W}$. The request can be denoted by ${\bf d} = (d_{1},...,d_{K})$, where user $k$ requests the file $W_{d_{k}}$, with $k\in \mathcal{K}$ and $d_{k}\in\{1,...,N\}$. On receiving the request ${\bf d}$, the server broadcasts the XOR multiplexing of packets $\bigoplus_{\substack{p_{j,k}=s,j\in\mathcal{F}\\k\in\mathcal{K}}}W_{d_{k},j}$ at each time slot $s\in \mathcal{S}$, thereby allowing each user to recover its requested file.
\end{itemize} The decoding algorithm per user is as follows. On requesting file $W_{d_{k}}\in \mathcal{W}$, user $k\in \mathcal{K}$ already has $\{W_{d_{k},j}\mid p_{j,k}=*\}$ in their cache. In order to recover $W_{d_{k}}$, the packets $\{W_{d_{k},j}\mid p_{j,k}\in\mathcal{S}\}$ must be decoded. By constraint (C), for each $s\in\mathcal{S}$, user $k$ also has the packets $\{W_{d_{k'},j}\mid p_{j,k}=s,k'\neq k\}$ in their cache at the placement phase. Therefore, the unknown values $W_{d_{k},j}$ for $p_{j,k}=s$ are then computed as the multiplexing difference \[\bigoplus_{\substack{p_{j,k}=s,j\in\mathcal{F}\\k\in\mathcal{K}}}W_{d_{k},j}-\bigoplus_{\substack{p_{j,k'}=s,j\in\mathcal{F}\\k'\in\mathcal{K},k'\neq k}}W_{d_{k'},j}.\]
\begin{example}\label{ex01} Here we present a $(4,4,2,4)$-PDA for a $(4,1,2)$ centralized coded caching scheme. Assume we are given two files. Divide each file into $F=4$ packets so that $W_{1}=\{W_{1,1},W_{1,2},W_{1,3},W_{1,4}\}$ and $W_{2}=\{W_{2,1},W_{2,2},W_{2,3},W_{2,4}\}$. Let $Z_{1},...,Z_{4}$ be the caches for the four users. Populate the users' caches as follows:\[
Z_{1}=Z_{3}=\{W_{1,1},W_{1,3},W_{2,1},W_{2,3}\}\text{ \ and \ }Z_{2}=Z_{4}=\{W_{1,2},W_{1,4},W_{2,2},W_{2,4}\}.\]It is not difficult to see that the following array is a $(4,4,2,4)$-PDA. \begin{equation}\label{eq01}
\mathcal{P}=\left[\begin{tabular}{cccc}
*&1&*&3\\
1&*&3&*\\
*&2&*&4\\
2&*&4&*\\\end{tabular}\right]
\end{equation}We display the contents that are broadcast for some of the 16 possibilities for ${\bf d}$ in Table I.
\begin{center}
\tabcolsep=0.7cm
\small
\begin{tabular}{c|cccc}
\hline
  \hline
Request&Time Slot 1&Time Slot 2&Time Slot 3&Time Slot 4\\\hline
$(1,1,1,1)$&$W_{1,2}\oplus W_{1,1}$&$W_{1,4}\oplus W_{1,3}$&$W_{1,2}\oplus W_{1,1}$&$W_{1,4}\oplus W_{1,3}$\\
$(1,1,1,2)$&$W_{1,2}\oplus W_{1,1}$&$W_{1,4}\oplus W_{1,3}$&$W_{1,2}\oplus W_{2,1}$&$W_{1,4}\oplus W_{2,3}$\\
$(1,1,2,1)$&$W_{2,2}\oplus W_{1,1}$&$W_{2,4}\oplus W_{1,3}$&$W_{1,2}\oplus W_{1,1}$&$W_{1,4}\oplus W_{1,3}$\\
$\vdots$&$\vdots$&$\vdots$&$\vdots$&$\vdots$\\
$(1,1,2,2)$&$W_{1,2}\oplus W_{1,1}$&$W_{1,4}\oplus W_{1,3}$&$W_{2,2}\oplus W_{2,1}$&$W_{2,4}\oplus W_{2,3}$\\
$(1,2,2,1)$&$W_{1,2}\oplus W_{2,1}$&$W_{1,4}\oplus W_{2,3}$&$W_{2,2}\oplus W_{1,1}$&$W_{2,4}\oplus W_{1,3}$\\
\end{tabular}
\captionof{table}{{\footnotesize Contents broadcast for some of the 16 possibilites ${\bf d}$.}}
\end{center} The decoding algorithm works as follows. Suppose that ${\bf d} = (1,2,2,1)$. To recover $W_{1}$ the first user can decode $W_{1,2}$ and $W_{1,4}$. The packets $W_{1,2}$ and $W_{1,4}$ can be computed by subtracting $W_{2,1}$ from $W_{1,2}\oplus W_{2,1}$, and $W_{2,3}$ from $W_{1,4}\oplus W_{2,3}$, respectively. To recover $W_{2}$ the second user can decode $W_{2,1}$ and $W_{2,3}$. The packets $W_{2,1}$ and $W_{2,3}$ can be computed by subtracting $W_{1,2}$ from $W_{1,2}\oplus W_{2,1}$, and $W_{1,4}$ from $W_{1,4}\oplus W_{2,3}$, respectively. The other two users' requests are decoded in a similar way.
\end{example}
For some good examples of PDAs where the decoding process is shown in detail, the reader is referred to \cite{CHONG}, \cite{YAN0} and \cite{YAN1}.
\subsection{Placement Delivery Arrays from Strong Edge Colorings}\label{sec2.1}
We will assume some familiarity with basic concepts of graph theory. For a graph $\Gamma=(V,E)$, a (proper) edge coloring of $\Gamma$ is an assignment of colors to the set of edges $E$ such that no two adjacent edges have the same color. The minimal number of colors needed in an edge coloring of a graph $\Gamma$ is referred to as the chromatic index of $\Gamma$, and is denoted by $\chi'(\Gamma)$. An edge coloring of $\Gamma$ such that no two edges of the same color are adjacent to a same edge is referred to a {\it strong edge coloring}. The smallest number of colors needed in a strong edge coloring is called the {\it strong chromatic index} of $\Gamma$, and is denoted by $sq(\Gamma)$. A (proper) vertex coloring of $\Gamma$ is an assignment of colors to the vertices of $\Gamma$ in such a way that no two adjacent vertices have the same color. The minimum number of colors needed in a vertex coloring of $\Gamma$ is called the {\it chromatic number}, and is denoted by $\chi(\Gamma)$. For example, the complete graph $K_{n}$ on $n$ vertices has strong chromatic index $sq(K_{n})=\binom{n}{2}$ and chromatic number $\chi(K_{n})=n$, and the complete bipartite graph $K_{n,m}$ with one vertex component containing $n$ vertices and the other containing $m$ vertices, has strong chromatic index $sq(K_{n,m})=nm$ and chromatic number $\chi(K_{n,m})=2$.
\par
An edge coloring $\mathcal{E}=\mathcal{E}(\Gamma)$ on $\Gamma=(V,E)$ with colors from the set $\mathcal{S}$ is a system of ordered pairs $\{(\{u,v\},s)\in E\times\mathcal{S}\mid \{u,v\}\in E\text{ and }\{u,v\}\text{ has color }s\}$. For the remainder of the paper, when a set $\mathcal{S}$ has been designated as the set of colors of an edge coloring $\mathcal{E}$, it is always assumed to be proper in the sense that each member of $\mathcal{S}$ corresponds to at least one member of $E$.
\par
This paper will be concerned mostly with bipartite graphs, which are those graphs $\Gamma=(V,E)$ such that $V=V_{1}\cup V_{2}$ where $V_{1}\cap V_{2}=\emptyset$, and $E\subseteq V_{1}\times V_{2}$. For convenience, we will simply write such a graph $\Gamma$ as $(V_{1},V_{2},E)$ to indicate that it is bipartite. Thus, an edge coloring $\mathcal{E}=\mathcal{E}(\Gamma)$ on $\Gamma=(V_{1},V_{2},E)$, with colors from the set $\mathcal{S}$ can be treated as the triple system $\{(v_{1},v_{2},s)\in V_{1}\times V_{2}\times\mathcal{S}\mid (v_{1},v_{2})\in E\text{ and }(v_{1},v_{2})\text{ has color }s\}$.
\par
It is easy to see that an edge-colored bipartite graph $\Gamma=(\mathcal{F},\mathcal{K},E)$ with colors from the set $\mathcal{S}$ can be viewed as an $F\times K$ array ${\bf A}=\left[a_{j,k}\right]$ where $a_{j,k}=*$ if $(j,k)\not\in E$ and $a_{j,k}=s\in\mathcal{S}$. Then the corresponding triple system would be given by \[\mathcal{E}(\Gamma)=\{(j,k,s)\in \mathcal{F}\times\mathcal{K}\times\mathcal{S}\mid (j,k)\in E \text{ and } (j,k) \text{ has color }s\}.\] The following was shown in \cite{YAN1} (and indirectly in \cite{CHONG}, in terms of $3$-uniform hypergraphs).
\begin{lemma}\label{le1} {\rm \cite{YAN1}} For any $F\times K$ array ${\bf A}=\left[a_{j,k}\right]$ composed of a symbol `` * '' and integers $1,...,S$, ${\bf A}$ is a PDA if and only if its corresponding edge-colored bipartite graph $\Gamma=(\mathcal{F},\mathcal{K},E)$ satisfies:
\begin{enumerate}
\item[(i)] the vertices in $\mathcal{K}$ have constant degree, and
\item[(ii)] the corresponding edge coloring $\mathcal{E}(\Gamma)$ is a strong edge coloring.
\end{enumerate}
\end{lemma}
\begin{example} Take the $(4,4,2,4)$-PDA $\mathcal{P}$ constructed in Example \ref{ex01}. Let $\mathcal{F}=\{a,b,c,d\}$ correspond to the rows of (\ref{eq01}), $\mathcal{K}=\{1',2',3',4'\}$ correspond to the columns, and $\mathcal{S}=\{1,2,3,4\}$ be the set of integers appearing in the array. We may let the members $1,2,3,4$ of $\mathcal{S}$ correspond to the colors red, blue, green and black, respectively. It is easy to see that this PDA has corresponding strong edge-colored bipartite graph given in Fig. \ref{fig1}.
\begin{figure}[!htb]
\begin{center}
\includegraphics[scale=0.5]{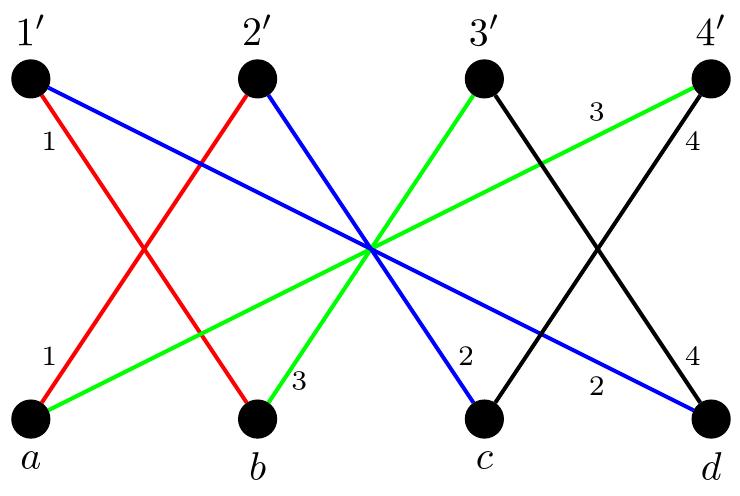}
\caption{Edge-colored bipartite graph corresponding to $\mathcal{P}$}
\label{fig1}
\end{center}
\end{figure}
\end{example}
\par We note here that to illustrate our constructions we will often use examples involving subset graphs, such as those discussed in Construction I of \cite{CHONG}, due to the flexibility of their parameters. We also remark that the examples we give to illustrate our constructions are in no way an exhaustive characterization of their uses. There are many known PDAs for which it may be worthwhile checking whether new PDAs with good parameters can be obtained after combining them with others via the methods discussed in this paper.
\section{A construction via combining two PDAs having the same set of integer entries}\label{sec3}
In this section we discuss a construction method that combines two PDAs that have the same set of integer entries to obtain a new PDA. This is equivalent (via Lemma \ref{le1}) to the construction of a strong edge coloring of a bipartite graph via combining two strong edge colorings of bipartite graphs that have the same set of colors.\par
The following theorem is the main the result of this section.
\begin{theorem}\label{co3.0} Suppose there exist PDAs $\mathcal{P}_{1}=\left[p_{i,j}^{1}\right]$ and $\mathcal{P}_{2}=\left[p_{i,j}^{2}\right]$ with parameters $(K_{1},F_{1},Z_{1},S')$ and $(K_{2},F_{2},Z_{2},S')$ respectively. Let $\mathcal{F}_{1}=\{1,...,F_{1}\},\mathcal{F}_{2}=\{1,...,F_{2}\},\mathcal{K}_{1}=\{1,...,K_{1}\}$ and $\mathcal{K}_{2}=\{1,...,K_{2}\}$. Then there exists a PDA $\mathcal{P}$ with parameters $(K,F,Z,S)$ where $K=|\{(i,i')\in \mathcal{F}_{1}\times\mathcal{F}_{2}\mid p_{i,j}^{1}=p_{i',j'}^{2}(\neq *)\text{ for some }(j,j')\in\mathcal{K}_{1}\times\mathcal{K}_{2}\}|,F=K_{1},Z=K_{1}-|\{j\in K_{1}\mid p_{i,j}^{1}=p_{i',j'}^{2}(\neq*)\text{ for some }j'\in K_{2}\}|,$ and $S=S'K_{2}$.
\end{theorem}
Here we give a simple illustrative example.
\begin{example} Let $\mathcal{P}_{1}$ and $\mathcal{P}_{2}$ both be $(4,2,1,2)$ PDAs so that \[\mathcal{P}_{1}=\left[\begin{tabular}{cccc}
*&1&*&2\\
1&*&2&*\\\end{tabular}\right]_{F_{1}\times K_{1}} \ \text{ and } \ \mathcal{P}_{2}=\left[\begin{tabular}{cccc}
*&1&*&2\\
1&*&2&*\\\end{tabular}\right]_{F_{2}\times K_{2}}.\] Let $\mathcal{K}_{1}=\mathcal{K}_{2}=\{1',2',3',4'\},\mathcal{F}_{1}=\mathcal{F}_{2}=\{1,2\}$ and $\mathcal{S}_{1}=\mathcal{S}_{2}=\{1,2\}$. Let $\Gamma_{1}=(\mathcal{F}_{1},\mathcal{K}_{1},E_{1})$ resp. $\Gamma_{2}=(\mathcal{F}_{2},\mathcal{K}_{2},E_{2})$ be the underlying bipartite graphs of $\mathcal{P}_{1}$ resp. $\mathcal{P}_{2}$, and let $\mathcal{E}_{1}$ resp. $\mathcal{E}_{1}$ be the corresponding strong edge colorings on $\Gamma_{1}$ resp. $\Gamma_{2}$. Then it is clear that $\mathcal{E}_{1}=\mathcal{E}_{2}=\{(1,2',1),(1,4',2),(2,1',1),(2,3',2)\}$, and $\Gamma_{1}$ and $\Gamma_{2}$ (together with their edge colorings) both have the form given in Fig. \ref{fig2}.
\begin{figure}[t]
\begin{minipage}[c]{0.45\textwidth}
\begin{center}
\includegraphics[scale=0.35]{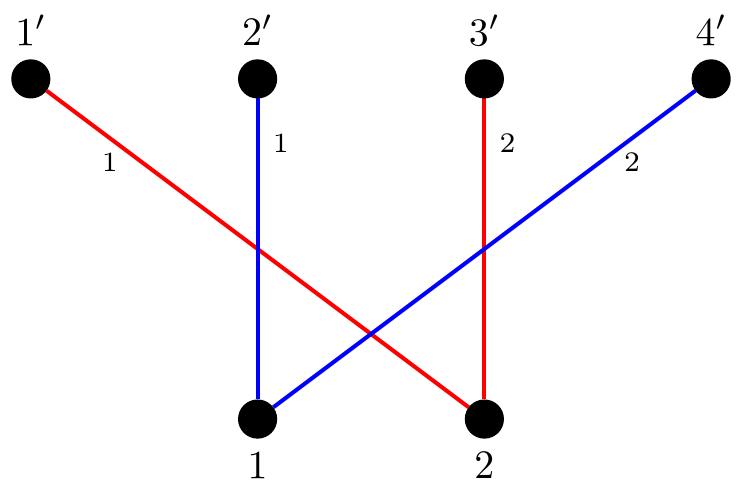}
\caption{Edge-colored $\Gamma_{1}$ (and $\Gamma_{2}$)}
\label{fig2}
\end{center}
\end{minipage}%
\hspace{0.05\textwidth}%
\begin{minipage}[c]{0.45\textwidth}
\begin{center}
\includegraphics[scale=0.35]{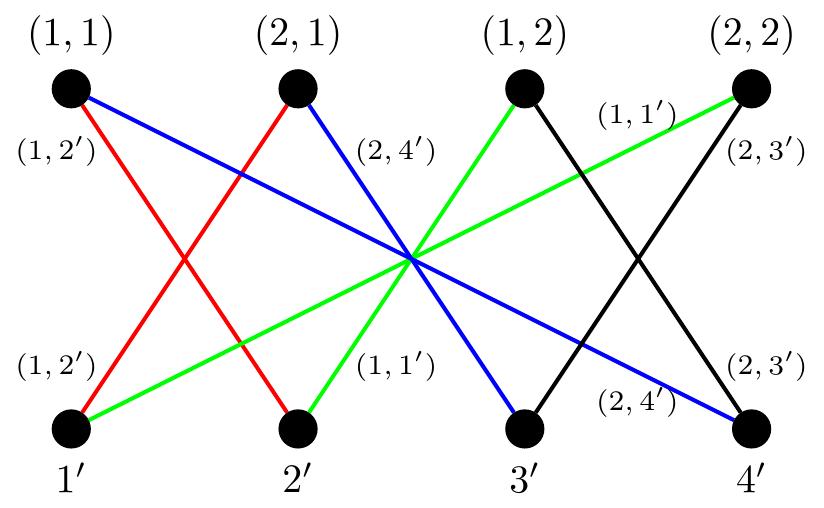}
\caption{Edge-colored $\Gamma$}
\label{fig3}
\end{center}
\end{minipage}
\end{figure}
 Notice that the set \[\begin{split}\mathcal{E}
 &=\{(y,(x,u),(s,v))\mid(x,y,s)\in \mathcal{E}_{1}\text{ and }(u,v,s)\in\mathcal{E}_{2}\}\\
 &=\{(2',(1,1),(1,2')),(2',(1,2),(1,1')),(1',(2,1),(1,2')),
 (1',(2,2),(1,1')),\\
 &\quad\quad(4',(1,1),(2,4')),(4',(1,2),(2,3')),(3',(2,1),(2,4')),(3',(2,2),(2,3'))\}\end{split}\] is a strong edge coloring on $\Gamma=(V,W,E)$, with colors from $\Sigma=\{(1,1'),(1,2'),(2,3'),(2,4')\}$, where
$V = \{1',2',3',4'\},W = \{(1,1),(1,2),(2,1),(2,2)\}$, and 
 \[E = \{(2',(1,1)),(2',(1,2)),(1',(2,1)),
 (1',(2,2)),(4',(1,1)),(4',(1,2)),(3',(2,1)),(3',(2,2))\}.\] The graph $\Gamma$ (together with its edge coloring) has the form given in Fig. \ref{fig3}.
Taking $\mathcal{K}=V,\mathcal{F}=W$ and $S=\Sigma$, with some relabelling, we obtain the same $(4,4,2,4)$ PDA $\mathcal{P}$ as was discussed in Example \ref{ex01}.
\end{example}
To prove Theorem \ref{co3.0} we will need the following lemma.
\begin{lemma}\label{th3.0} Let $\Gamma_{1}=(V_{1},W_{1},E_{1})$ and $\Gamma_{2}=(V_{2},W_{2},E_{2})$ be two bipartite graphs, and let $\mathcal{E}_{1}$ and $\mathcal{E}_{2}$ be strong edge colorings on $\Gamma_{1}$ and $\Gamma_{2}$, respectively, both having colors from $\mathcal{S}$. Define \begin{eqnarray*}
V & = & W_{1}\\
W & = & \{(x,u)\in V_{1}\times V_{2}\mid(x,y,s)\in\mathcal{E}_{1}\text{ and }(u,v,s)\in\mathcal{E}_{2}\text{ for some }y\in W_{1},v\in W_{2},s\in\mathcal{S}\},\\
E & = & \{(y,(x,u))\in V\times W\mid(x,y,s)\in \mathcal{E}_{1}\text{ and }(u,v,s)\in\mathcal{E}_{2}\text{ for some }s\in\mathcal{S},v\in W_{2}\}.\end{eqnarray*} Then \[
\mathcal{E}=\{(y,(x,u),(s,v))\mid(x,y,s)\in \mathcal{E}_{1}\text{ and }(u,v,s)\in\mathcal{E}_{2}\}\] is a strong edge coloring on the bipartite graph $\Gamma=(V,W,E)$ with colors from $\Sigma=\{(s,v)\in \mathcal{S}\times W_{2}\mid(u,v,s)\in\mathcal{E}_{2}\text{ for some }u\in V_{2}\}$.
\end{lemma}
\begin{proof} That $\Gamma$ is bipartite is straight forward. We first show that $\mathcal{E}$ is an edge coloring of $\Gamma$. Suppose that $(y,(x,u),(s,v)),(y',(x',u'),(s,v))\in\mathcal{E}$. First note we must have $u=u'$ since $(u,v,s),(u',v,s)\in\mathcal{E}_{2}$ and $\mathcal{E}_{2}$ is an edge coloring. If $y=y'$ then $(x,y,s),(x',y,s)\in\mathcal{E}_{1}$. Then we must have $x=x'$ since $\mathcal{E}_{1}$ is an edge coloring. Also notice that, if $(x,u)=(x',u')$, then we must have $y=y'$ since $(x,y,s),(x,y',s)\in\mathcal{E}_{1}$ and $\mathcal{E}_{1}$ is an edge coloring. Thus $\mathcal{E}$ is an edge coloring.
\par
We now show $\mathcal{E}$ is a strong edge coloring. Suppose that $(y,(x,u),(s,v)),(y',(x',u'),(s,v))\in\mathcal{E}$. We show that there are no members $(y,(x',u'),(s',v'))$ or $(y',(x,u),(s',v'))$ in $\mathcal{E}$ for any $(s',v')\in \Sigma$. Note again that we must have $u=u'$ since 
$(u,v,s),(u',v,s)\in \mathcal{E}_{2}$ and $\mathcal{E}_{2}$ is an edge coloring. Also note that we must have $x\neq x'$ since, if $x=x'$, then $(x,y,s),(x,y',s)\in\mathcal{E}_{1}$ contradicting that $\mathcal{E}_{1}$ is an edge coloring.
\par
Now suppose that $(y,(x',u'),(s',v'))\in \mathcal{E}$ for some $(s',v')\in\Sigma$. By definition of $\mathcal{E}$ we have that $(x',y,s')\in\mathcal{E}_{1}$. If $s=s'$ then we have $(x',y,s),(x,y,s)\in\mathcal{E}_{1}$. Since $x'\neq x$, this contradicts that $\mathcal{E}_{1}$ is an edge coloring. If $s'\neq s$ then we have $(x,y,s),(x',y',s),(x',y,s')\in\mathcal{E}_{1}$ which contradicts that $\mathcal{E}_{1}$ is a strong edge coloring. Thus there is no member $(y,(x',u'),(s',v'))\in \mathcal{E}$ for any $(s',v')\in\Sigma$.
\par
Now suppose that $(y',(x,u),(s',v'))\in \mathcal{E}$ for some $(s',v')\in\Sigma$. By definition of $\mathcal{E}$ we have that $(x,y',s')\in\mathcal{E}_{1}$. If $s=s'$ then we have $(x,y,s),(x,y',s)\in\mathcal{E}_{1}$ which contradicts that $\mathcal{E}_{1}$ is an edge coloring. If $s'\neq s$ then we have $(x,y,s),(x',y',s),(x,y',s')\in\mathcal{E}_{1}$ which contradicts that $\mathcal{E}_{1}$ is a strong edge coloring. Thus there is no member $(y',(x,u),(s',v'))\in \mathcal{E}$ for any $(s',v')\in\Sigma$. This completes the proof.
\end{proof}\noindent
{\it Proof of Theorem \ref{co3.0}}. This is an immediate consequence of Lemma \ref{th3.0} and Lemma \ref{le1}. \qed\par
Here we give an example of how Lemma \ref{th3.0} can be used to construct a family of PDAs.
\begin{example}\label{ex3.0} Let $a,b,n$ and $t$ be positive integers such that $a+b\leq n$ and $0\leq t< b$. Let $\left[n\right]=\{1,...,n\}$ and let $\binom{\left[n\right]}{a}$ denote the collection of subsets of $\left[n\right]$ of size $a$. Define $\Gamma_{1}=(V_{1},W_{1},E_{1})$ where $V_{1}=\binom{\left[n\right]}{a+t},W_{1}=\binom{\left[n\right]}{b-t}$ and $E_{1}=\{(A,B)\in V_{1}\times W_{1}\mid A\cap B=\emptyset\}$, and $\Gamma_{2}=(V_{2},W_{2},E_{2})$ where $V_{2}=\binom{\left[n\right]}{a},W_{2}=\binom{\left[n\right]}{b}$ and $E_{2}=\{(A,B)\in V_{2}\times W_{2}\mid A\cap B=\emptyset\}$. Then taking $\mathcal{C}=\binom{\left[n\right]}{a+b}$, we have that $\mathcal{E}_{1}=\{(A,B,C)\in V_{1}\times W_{1}\times\mathcal{C}\mid A\cup B=C\}$ and $\mathcal{E}_{2}=\{(A,B,C)\in V_{2}\times W_{2}\times\mathcal{C}\mid A\cup B=C\}$ are strong edge colorings on $\Gamma_{1}$ and $\Gamma_{2}$ respectively both having colors from $\mathcal{C}$. Let $V,W,E$ and $\Sigma$ be defined as in Lemma \ref{th3.0}, set $\Gamma=(V,W,E)$, and let $\mathcal{E}$ be the corresponding edge coloring with colors from $\Sigma$.
\par
For the purpose of making it easier to count the valency of the vertices, we construct a PDA from a subgraph of $\Gamma$ in the following way. Take $W^{*}=\{(A,A')\in W\mid A'\subseteq A\}$, and let $\Gamma^{*}=(V,W^{*},E^{*})$ be the graph obtained by removing any edges of $\Gamma$ which contain vertices in $W\setminus W^{*}$. Thus, if $(A,A')\in W^{*}$, then there are $\binom{n-a-t}{b-t}$ members of $\mathcal{C}$ containing $A'$ so that the vertices in $W^{*}$ have constant degree. If we let $\mathcal{E}^{*}$ be the subset of $\mathcal{E}$ obtained by removing any edges which are not in $\Gamma$, then we have that $\mathcal{E}^{*}$ is a strong edge coloring on $\Gamma^{*}$. Then, taking $\mathcal{F}=V,\mathcal{K}=W^{*}$ and $\mathcal{S}=\Sigma$, the PDA will have the following parameters.
\begin{center}
\tabcolsep=1.0cm
\small
\begin{tabular}{c|c|c|c}
\hline
  \hline
$K$&$1-M/N$&$F$&$R$\\\hline
$\binom{n}{a+t}\binom{a+t}{a}$&$\binom{n-a-t}{b-t}/\binom{n}{b-t}$&$\binom{n}{b-t}$&$\binom{n}{a+b}\binom{a+b}{b}/\binom{n}{b-t}$\\\hline
\end{tabular}
\end{center}
\par
If we take $a=1$ and $t=1$ then we obtain a $(2\binom{n}{2},\binom{n}{b-1},\binom{n}{b-1}-\binom{n-2}{b-1},\binom{n}{b+2})$-PDA. If we take $n=\lambda b$ for some constant $\lambda>1$, then we get $R=S/F=(b+1)\binom{\lambda b}{b+1}/\binom{\lambda b}{b-1}=b(\lambda-1)^{2}+(\lambda-1)$ and $1-M/N=1-Z/F=\binom{\lambda b-2}{b-1}/\binom{\lambda b}{b-1}\approx ((\lambda-1)/\lambda)^{2}$ (for $b$ large). By Stirling's formula we have that (for $n$ large) $K=2\binom{n}{2}\approx n^{2}$, and \[
F=\binom{n}{\lambda^{-1} n-1}\approx \frac{1+o(1)}{\sqrt{2\pi\lambda^{-1}(1-\lambda^{-1})n}}2^{nH(\lambda^{-1})}=\mathcal{O}(K^{-1/4}2^{\sqrt{K}H(\lambda^{-1})}), 
\] where \begin{equation}\label{eqa} H(x)=-x\log_{2}(x)-(1-x)\log_{2}(1-x)\end{equation} for $0<x<1$. Thus, $F$ still grows sub-exponentially with $K$, but at a slower rate than that of the initial (component) PDAs, whose growth rate can be found in Remark IV.3 of \cite{CHONG}
\end{example}
\begin{remark} It is obvious that the method for combining strong edge colorings put forth in this section can also be applied recursively to a family of three or more bipartite graphs, but characterizing such in closed form seems difficult as the adjacency relations quickly become complicated, and so this is left as an open problem.
\end{remark}
\section{A method for constructing PDAs based on tensor-like products of their underlying bipartite graphs}\label{sec4}
In this section we discuss a method for constructing PDAs that involves a tensor-like product of the PDAs' underlying bipartite graphs.\par
The following theorem is the main result of this section.
\begin{theorem}\label{th4.3} Let $l$ be a positive integer. Suppose there exist PDAs $\mathcal{P}_{i}$ with parameters $(K_{i},F_{i},Z_{i},S_{i})$, respectively, for $1\leq i\leq l$, where $Z_{i}=F_{i}-g_{i}$. Then there exists a PDA $\mathcal{P}$ with parameters $(K,F,Z,S)$ where $K=\Pi_{i=1}^{l}K_{i},F=\Pi_{i=1}^{l}F_{i},Z=\Pi_{i=1}^{l}F_{i}-\Pi_{i=1}^{l}g_{i}$ and $S=\Pi_{i=1}^{l}S_{i}$.
\end{theorem}To prove Theorem \ref{th4.3} we will need some results concerning strong edge colorings of tensor and tensor-like products of bipartite graphs.\par
If $\Gamma_{i}=(V_{i},E_{i})$, $1\leq i\leq l$, is a family of graphs, the tensor product $\Gamma_{1}\times\cdots\times\Gamma_{l}$ is the graph with vertex set $V=\prod_{i=1}^{l}V_{i}$ and edge set $E=\{((x_{1},...,x_{l}),(y_{1},...,y_{l}))\in V\times V\mid \{x_{i},y_{i}\}\in E_{i}\text{ for each }i\}$. It is not difficult to see that a tensor product is bipartite if and only if at least one of its factors is bipartite. The following is an immediate corollary to the proof of Theorem 4 of \cite{TOGN} (where the authors show that $sq(\Gamma_{1}\times\Gamma_{2})\leq sq(\Gamma_{1})sq(\Gamma_{2})$).
\begin{lemma}\label{th4.2} Let $\Gamma_{i}=(V_{i},E_{i})$, $1\leq i\leq l$, be a family of graphs, at least one of which is bipartite. Also, for each $i,1\leq i\leq l$, let $\mathcal{E}_{i}$ be a strong edge coloring of $\Gamma_{i}$ with colors from $S_{i}$. Then $\Gamma=\Gamma_{1}\times\cdots\times\Gamma_{l}$ is bipartite, and \[\mathcal{E}=\{((x_{1},...,x_{l}),(y_{1},...,y_{l}),(s_{1},...,s_{l}))\mid(x_{i},y_{i},s_{i})\in\mathcal{E}_{i}\text{ for each }i\}\] is a strong edge coloring on $\Gamma$ with colors from $\prod_{i=1}^{l}\mathcal{S}_{i}$.
\end{lemma}
It is possible to obtain new PDAs from strong edge colorings of other tensor-like products. If we let $\Gamma_{i}=(V_{i},W_{i},E_{i})$ for each $i,1\leq i\leq l$, then we can define the graph $\Gamma_{1}\star\cdots\star\Gamma_{l}$ to be the bipartite graph $(V,W,E)$ where $V=\prod_{i=1}^{l}V_{i}$, $W=\prod_{i=l}^{l}W_{i}$ and $E=\{((x_{1},...,x_{l}),(y_{1},...,y_{l}))\in V\times W\mid (x_{i},y_{i})\in E_{i}\text{ for each }i\}$. Note that $\Gamma_{1}\star\cdots\star\Gamma_{l}$ is a subgraph of the standard tensor product $\Gamma_{1}\times\cdots\times\Gamma_{l}$. Being such, the following lemma is easily deduced.
\begin{lemma}\label{th4.1} Let $\Gamma_{i}=(V_{i},W_{i},E_{i})$ be bipartite for each $i,1\leq i\leq l$. Also, for each $i,1\leq i\leq l$, let $\mathcal{E}_{i}$ be a strong edge coloring of $\Gamma_{i}$ with colors from $S_{i}$. Then \[\mathcal{E}=\{((x_{1},...,x_{l}),(y_{1},...,y_{l}),(s_{1},...,s_{l}))\mid(x_{i},y_{i},s_{i})\in\mathcal{E}_{i}\text{ for each }i\}\] is a strong edge coloring on $\Gamma=\Gamma_{1}\star\cdots\star\Gamma_{l}$ with colors from $\prod_{i=1}^{l}\mathcal{S}_{i}$.
\end{lemma}
\noindent
{\it Proof of Theorem \ref{th4.3}}. This is an immediate consequence of Lemma \ref{th4.1} and Lemma \ref{le1}. \qed\par
Here we give some examples.
\begin{example}\label{ex4.2} Let $n,a,b,t$ resp. $a',b',n',t'$ be positive integers such that $0<a,b<n$ and $0\leq t\leq \{a,b\}$, resp. $0<a',b'<n$ and $0\leq t'\leq \{a,b\}$. Define $\Gamma_{1}=(V_{1},W_{1},E_{1})$ where $V_{1}=\binom{\left[n\right]}{a},W_{1}=\binom{\left[n\right]}{b}$ and $E_{1}=\{(A,B)\in V_{1}\times W_{1}\mid |A\cap B|=t\}$, and $\Gamma_{2}=(V_{2},W_{2},E_{2})$ where $V_{2}=\binom{\left[n'\right]}{a'},W_{2}=\binom{\left[n'\right]}{b'}$ and $E_{2}=\{(A',B')\in V_{2}\times W_{2}\mid |A'\cap B'|=t'\}$. It was shown in \cite{YAN1} that, taking $\mathcal{C}_{1}=\{(D,I)\in\binom{\left[n\right]}{a+b-t}\times\binom{\left[n\right]}{t}\mid D\cap I=\emptyset\}$ and $\mathcal{C}_{2}=\{(D',I')\in\binom{\left[n'\right]}{a'+b'-t'}\times\binom{\left[n'\right]}{t'}\mid D'\cap I'=\emptyset\}$, we have that $\mathcal{E}_{1}=\{(A,B,C)\in V_{1}\times W_{1}\times\mathcal{C}_{1}\mid A\setminus B\cup B\setminus A=D\text{ and }A\cap B=I\}$ and $\mathcal{E}_{2}=\{(A',B',C')\in V_{2}\times W_{2}\times\mathcal{C}_{2}\mid A'\setminus B'\cup B'\setminus A'=D'\text{ and }A'\cap B'=I'\}$ are strong edge colorings on $\Gamma_{1}$ and $\Gamma_{2}$, respectively, with colors from $\mathcal{C}_{1}$ and $\mathcal{C}_{2}$. Then, by Lemma \ref{th4.1} we have that $\mathcal{E}=\{((A,A'),(B,B'),(C,C'))\mid (A,B,C)\in \mathcal{E}_{1}\text{ and }(A',B',C')\in \mathcal{E}_{2}\}$ is a strong edge coloring on $\Gamma=\Gamma_{1}\star\Gamma_{2}$.
\par
A PDA can be constructed from $\Gamma$ by taking $\mathcal{F}=V_{1}\times V_{2},\mathcal{K}=W_{1}\times W_{2}$ and $\mathcal{S}=\mathcal{C}_{1}\times \mathcal{C}_{2}$. Clearly the vertices in $W_{1}\times W_{2}$ have constant degree $\binom{b}{t}\binom{b'}{t'}\binom{n-b}{a-t}\binom{n'-b'}{a'-t'}$. Then the PDA will have the following parameters.
\begin{center}
\small
\begin{tabular}{c|c|c|c}
\hline
  \hline
$K$&$1-M/N$&$F$&$R$\\\hline
$\binom{n}{b}\binom{n'}{b'}$&$\binom{b}{t}\binom{b'}{t'}\binom{n-b}{a-t}\binom{n'-b'}{a'-t'}/(\binom{n}{a}\binom{n'}{a'})$&$\binom{n}{a}\binom{n'}{a'}$&$\binom{n}{a+b-2t}\binom{n-(a+b-2t)}{t}\binom{n'}{a'+b'-2t'}\binom{n'-(a'+b'-2t')}{t'}/(\binom{n}{a}\binom{n'}{a'})$\\\hline
\end{tabular}
\end{center}
If we take $n=n',a=a',b=b'=2$ and $t=t'=1$ then we obtain a $(\binom{n}{2}^{2},\binom{n}{a}^{2},\binom{n}{a}^{2}-4\binom{n-2}{a-1}^{2},(n-a)^{2}\binom{n}{a}^{2})$-PDA. If we take $n=\lambda a$ for some constant $\lambda>1$, then we get $R=S/F=((n-a)\binom{n}{a}/\binom{n}{a})^{2}=(n-a)^{2}$ and $1-M/N=1-Z/F=(2\binom{n-2}{a-1}/\binom{n}{a})^{2}\approx ((2\lambda-2)/\lambda^{2})^{2}$ (for $n$ large). By Stirling's formula we have that (for $n$ large) $K=\binom{n}{2}^{2}\approx n^{4}/4$, and \[
F=\binom{n}{\lambda^{-1}n}^{2}\approx \frac{1+o(1)}{2\pi\lambda^{-1}(1-\lambda^{-1})n}2^{2nH(\lambda^{-1})}=\mathcal{O}(K^{-1/4}2^{2\sqrt{2}K^{1/4}H(\lambda^{-1})}), 
\] where $H$ is defined as in Equation (\ref{eqa}). Note that although the rate at which $F$ grows with $K$ is still sub-exponential, it is an improvement from that of the initial (component) PDA, whose growth rate can be found in Section V of \cite{YAN1}.
\end{example}
\begin{example}\label{ex4.0} Let $\mathcal{P}$ be a $(K,F,Z,S)$ PDA and let $\Gamma_{1}=(V_{1},W_{1},E_{1})$ resp. $\mathcal{E}_{1}$ be its underlying bipartite graph resp. strong edge edge coloring with colors from $\mathcal{S}_{1}$. Note that $\mathcal{P}$ yields a caching scheme for $K$ users with ratio $M/N=Z/F$ and rate $R=S/F$. Now let $\Gamma_{2}=(V_{2},W_{2},E_{2})=K_{1,m}$ where $V_{2}=\{1\}$ and $W_{2}=\{1,...,m\}$. Since any strong coloring on $\Gamma_{2}$ must be such that no two edges can have the same color, we can take $\mathcal{E}_{2}=E_{2}$ to be the stong edge coloring on $\Gamma_{2}$. By Lemma \ref{th4.1} we have that $\mathcal{E}=\{(x,1),(y,v),s)\mid (x,y,s)\in\mathcal{E}_{2}\text{ and }(1,v)\in E_{2}\}$ is a strong edge coloring on $\Gamma=\Gamma_{1}\star\Gamma_{2}$. The correspondng PDA has parameters $(mK,F,Z,mS)$ and yields a caching scheme for $mK$ users with ratio $M/N=Z/F$ and rate $R=mS/F$. This is in fact the same PDA and caching scheme as one would obtain by using the simple grouping method introduced in \cite{SHANM}.
\end{example}
\section{A method for constructing PDAs based on strong-like products of their underlying bipartite graphs}\label{sec5}
In this section we discuss a method for constructing PDAs that involves a strong-like product of the PDAs' underlying bipartite graphs.\par The following theorem is the main result of this section.
\begin{theorem}\label{th5.2} Suppose there exists a PDA $\mathcal{P}'$ with parameters $(K',F',Z',S')$ where $Z'=F'-g'$. Let $m$ be a positive integer divisble by $6$. Then there exists a PDA $\mathcal{P}$ with parameters $(K,F,Z,S)$ where $K=mK',F=mF',Z=mF'-3g'$ and $S=8S'$.
\end{theorem} First we will need some terminology.\par
An arbitrary undirected (not necessarily bipartite) graph $\Gamma=(V,E)$ can be viewed as the directed graph $D(\Gamma)=(V,D(E))$ where $D(E)=\{\langle x,y\rangle,\langle y,x\rangle\mid \{x,y\}\in E\}$, and where $\langle x,y\rangle$ denotes the directed edge $x\rightarrow y$ in $D(\Gamma)$ (i.e., $|D(E)|=2|E|$). For an edge coloring $\mathcal{E}$ of $\Gamma$ with colors from $\mathcal{S}$, we say that $\mathcal{E}^{\circ}\subseteq D(E)\times\mathcal{S}$ is an {\it orientation} of $\mathcal{E}$ if either $(\langle x,y\rangle,s)$ or $(\langle y,x\rangle,s)$ (but not both) is a member of $\mathcal{E}^{\circ}$ whenever $(\{x,y\},s)\in \mathcal{E}$. The {\it opposing orientation} $\mathcal{E}_{\circ}$ of $\mathcal{E}^{\circ}$ is given by $\{(\langle x,y\rangle,s)\mid(\langle y,x\rangle,s)\in\mathcal{E}^{\circ}\}$. Thus we may speak of two opposing orientations of an edge coloring that have mutually disjoint sets of colors of equal cardinality, in which case it is understood that there is a bijection between the orientations preserving undirected edges and color assignment. For example, consider $\Gamma=(V,W,E)$ where $V=\{a,b,c\},W=\{x,y,z\}$, and $E=\{\{a,y\},\{a,z\},\{b,x\},\{b,z\},\{c,x\},\{c,y\}\}$. Then $\mathcal{E}=\{(a,y,1),(a,z,2),(b,x,1),(b,z,3),(c,x,2),(c,y,3)\}$ is a strong edge coloring on $\Gamma$ with colors from $\mathcal{S}=\{1,2,3\}$. Then two opposing orientations of $\mathcal{E}$ could be \begin{eqnarray*}
\mathcal{E}^{\circ} & = & \{(\langle a,y\rangle,1),(\langle a,z\rangle,2),(\langle b,x\rangle,1),(\langle b,z\rangle,3),(\langle c,x\rangle,2),(\langle c,y\rangle,3)\},\text{ and}\\
\mathcal{E}_{\circ} & = & \{(\langle y,a\rangle,1),(\langle z,a\rangle,2),(\langle x,b\rangle,1),(\langle z,b\rangle,3),(\langle x,c\rangle,2),(\langle y,c\rangle,3)\},\end{eqnarray*} and two opposing orientations with mutually disjoint sets of colors could be \begin{eqnarray*}\mathcal{E}^{\circ} & = & \{(\langle a,y\rangle,1),(\langle a,z\rangle,2),(\langle b,x\rangle,1),(\langle b,z\rangle,3),(\langle c,x\rangle,2),(\langle c,y\rangle,3)\},\text{ and}\\
\mathcal{E}_{\circ} & = & \{(\langle y,a\rangle,1'),(\langle z,a\rangle,2'),(\langle x,b\rangle,1'),(\langle z,b\rangle,3'),(\langle x,c\rangle,2'),(\langle y,c\rangle,3')\}\end{eqnarray*} (the mutually disjoint sets of colors being $\mathcal{S}^{\circ}=\{1,2,3\}$ and $\mathcal{S}_{\circ}=\{1',2',3'\}$).\par
Here we give an illustrative example.
\begin{example}\label{ex5.2} Let $\Gamma_{1}=(V_{1},E_{1})=C_{3}$ be the cycle on three vertices where $V_{1}=\{1,2,3\}$ and $E_{1}=\{\{1,2\},\{2,3\},\{3,1\}\}$. Let $\mathcal{E}_{1}=\{(\{1,2\},1),(\{2,3\},2),(\{3,1\},3)\}$ be a strong edge coloring on $\Gamma_{1}$ with colors from $\mathcal{S}^{\circ}=\{1,2,3\}$, and let $\mathcal{V}=\{(1,a),(2,b),(3,c)\}$ be a vertex coloring with colors from $\mathcal{S}'=\{a,b,c\}$.\par
For illustrative purposes, let $\mathcal{P}$ be the trivial $(2,2,1,2)$ PDA given by \[
\mathcal{P}=\left[\begin{tabular}{cc}
*&1\\
1&*\\\end{tabular}\right]_{F\times K},\] which yields a caching scheme with $M/N=Z/F=1/2$ and $R=S/F=1/2$. Take $\mathcal{K}=\{1',2'\},\mathcal{F}=\{1,2\}$ and $\mathcal{S}_{2}=\{1\}$. Let $\Gamma_{2}=(\mathcal{F},\mathcal{K},E_{2})$ resp. $\mathcal{E}_{2}=\{(1,2',1),(2,1',1)\}$ be the underlying bipartite graph resp. corresponding strong edge coloring of $\mathcal{P}$. Then $\Gamma_{1}$ and $\Gamma_{2}$ (together with their edge-colorings) have the forms given in Fig. \ref{fig4} and Fig. \ref{fig5} respectively.
\begin{figure}[b]
\begin{minipage}[c]{0.45\textwidth}
\begin{center}
\includegraphics[scale=0.35]{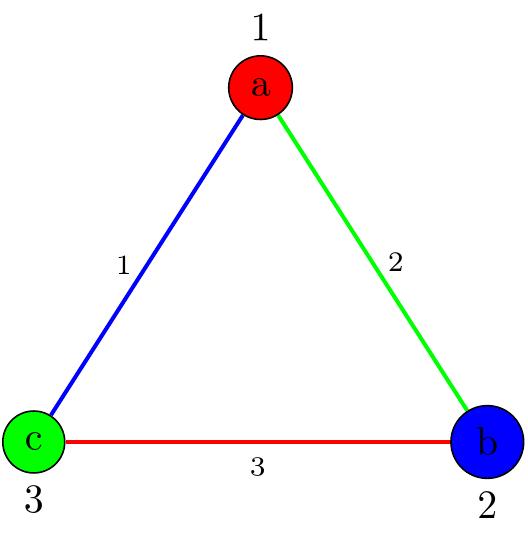}
\caption{Edge (and vertex)-colored $\Gamma_{1}$}
\label{fig4}
\end{center}
\end{minipage}%
\hspace{0.05\textwidth}%
\begin{minipage}[c]{0.35\textwidth}
\begin{center}
\includegraphics[scale=0.35]{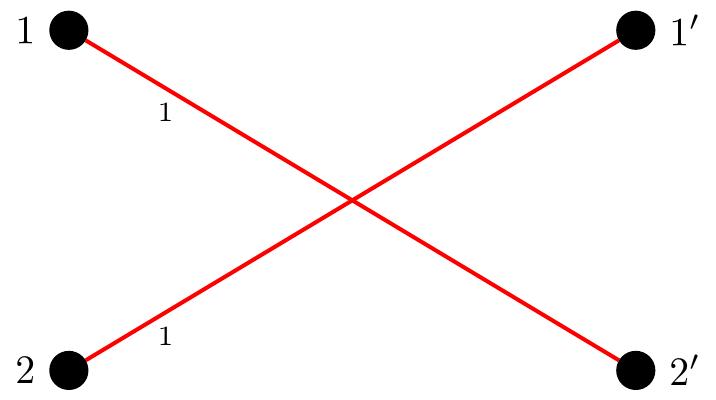}
\vspace{20px}
\caption{Edge-colored $\Gamma_{2}$}
\label{fig5}
\end{center}
\end{minipage}
\end{figure}
resectively. Consider the graph $\Gamma=(V,W,E)$ where $V=V_{1}\times\mathcal{F},W=V_{1}\times\mathcal{K}$ and $E=\{((x,u),(y,v)\in V\times W\mid (y,v)\in E_{2}\text{ and }x=u\text{ or }(x,u)\in E_{1}\}$. Note that $\Gamma$ has the form given in Fig. \ref{fig6}.
Let $\mathcal{E}^{\circ}=\{(\langle 1,2\rangle,2),(\langle 2,3\rangle,3),(\langle 3,1\rangle,1)\}$ and $\mathcal{E}_{\circ}=\{(\langle 2,1\rangle,2'),(\langle 3,2\rangle,3'),(\langle 1,3\rangle,1')\}$ be two opposing orientations of $\mathcal{E}_{1}$ with colors from $\mathcal{S}^{\circ}$ and $\mathcal{S}_{\circ}$ respectively. Notice that \[\begin{split}\mathcal{E}&=\{((x,y),(u,v),(s,s_{2}))\mid(y,v,s_{2})\in\mathcal{E}_{2};\text{ and }x=u\text{ with }(x,s)\in\mathcal{V},\text{ or }(\langle x,u\rangle,s)\in\mathcal{E}^{\circ}\cup\,\mathcal{E}_{\circ}\}\\
&=\{((1,1),(1,2'),(a,1)),((1,1),(2,2'),(2,1)),((1,1),(3,2'),(1,1)),((1,2),(1,1'),(a,1)),\\
& \quad\quad ((1,2),(2,1'),(2,1)),((1,2),(3,1'),(1,1)),((2,1),(1,2'),(2',1)),((2,1),(2,2'),(b,1)),\\
&\quad\quad((2,1),(3,2'),(3',1)),((2,2),(1,1'),(2',1)),((2,2),(2,1'),(b,1)),((2,2),(3,1'),(3,1)),\\
&\quad\quad((3,1),(1,2'),(1',1)),((3,1),(2,2'),(3',1)),((3,1),(3,2'),(c,1)),((3,2),(1,1'),(2',1)),\\
&\quad\quad((3,2),(2,1'),(3',1)),((3,2),(3,1'),(c,1))\}\end{split}\] is a strong edge coloring on $\Gamma$ with colors from $\mathcal{S}=(\mathcal{S}'\cup\mathcal{S}^{\circ}\cup\mathcal{S}_{\circ})\times\mathcal{S}_{2}$. Taking $\mathcal{K}=W,\mathcal{F}=V$ and $\mathcal{S}=\mathcal{S}$, by Lemma \ref{le1}, $\mathcal{E}$ corresponds to a $(6,6,3,9)$ PDA, which yields a caching scheme with $M/N=Z/F=1/2$ and $R=S/F=3/2$.
\end{example}
Note that, although the construction in Example \ref{ex5.2} is not exactly an instance of the construction in Theorem \ref{th5.2}, the idea behind it is the same. To prove Theorem \ref{th5.2} we will need some results about strong-like products of bipartite graphs.
\par
Let $\Gamma_{1}=(V_{1},E_{1})$ be a graph, and $\Gamma_{2}=(V_{2},W_{2},E_{2})$ be a bipartite graph. We define $\Gamma_{1}:\Gamma_{2}$ to be the bipartite graph $(V_{1}\times V_{2},V_{1}\times W_{2},E)$ where $E=\{((x,y),(u,v))\mid(y,v)\in E_{2};\text{ and }x=u\text{ or } \{x,u\}\in E_{1}\}$. In this section we will treat a vertex coloring $\mathcal{V}=\mathcal{V}(\Gamma)$ of an arbitrary graph $\Gamma=(V,E)$, with colors from $\mathcal{S}'$, as a set of ordered pairs $\{(v,s)\in V\times\mathcal{S}'\mid v\text{ has color }s\}$. We will need the following lemma.
\begin{figure}[h]
\begin{center}
\includegraphics[scale=0.35]{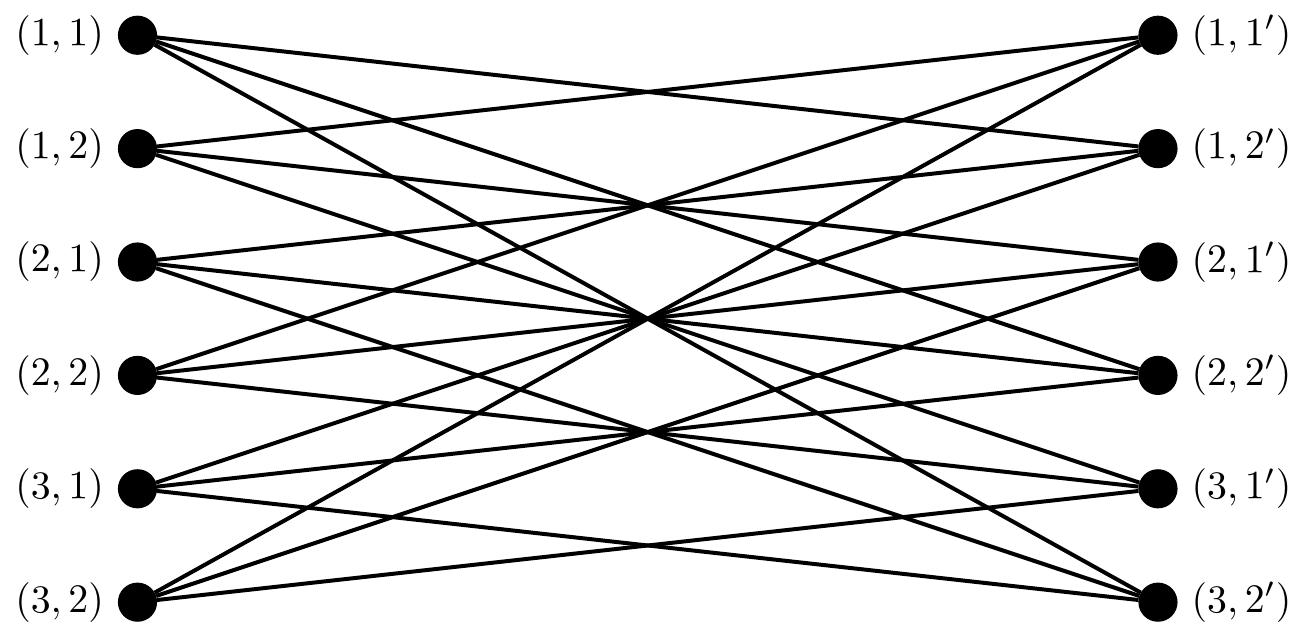}
\caption{$\Gamma$}
\label{fig6}
\end{center}
\end{figure}
\begin{lemma}\label{th5.1} Let $\Gamma_{1}=(V_{1},E_{1})$ be a graph with strong edge coloring $\mathcal{E}_{1}$, and $\Gamma_{2}=(V_{2},W_{2},E_{2})$ be a bipartite graph with strong edge coloring $\mathcal{E}_{2}$ and colors from $\mathcal{S}_{2}$. Let $\mathcal{E}^{\circ}$ and $\mathcal{E}_{\circ}$ be two opposing orientations of $\mathcal{E}_{1}$ with mutually disjoint sets of colors $\mathcal{S}^{\circ}$ and $\mathcal{S}_{\circ}$, respectively, and suppose there exists a proper vertex coloring $\mathcal{V}$ of $\Gamma_{1}$ with colors from $\mathcal{S}'$ (where $\mathcal{S}^{\circ}\cap\mathcal{S}'=\mathcal{S}_{\circ}\cap\mathcal{S}'=\emptyset$). Then \[\mathcal{E}=\{((x,y),(u,v),(s,s_{2}))\mid(y,v,s_{2})\in\mathcal{E}_{2};\text{ and }x=u\text{ with }(x,s)\in\mathcal{V},\text{ or }(\langle x,u\rangle,s)\in\mathcal{E}^{\circ}\cup\mathcal{E}_{\circ}\}\] is a strong edge coloring on $\Gamma=\Gamma_{1}:\Gamma_{2}$ with colors from $\mathcal{S}=(\mathcal{S}'\cup\mathcal{S}^{\circ}\cup\mathcal{S}_{\circ})\times \mathcal{S}_{2}$.
\end{lemma}
\begin{proof} Note that, without loss of generality, we may take the set of colors of $\mathcal{E}_{1}$ to be $\mathcal{S}^{\circ}$. We first show that $\mathcal{E}$ is an edge coloring on $\Gamma$. Let $((x,y),(u,v),(s,s_{2})),((x',y'),(u',v'),(s,s_{2}))\in\mathcal{E}$ agree in exactly two coordinates. Assume $(x,y)=(x',y')$. If $x=u$ then $(x,s)\in\mathcal{V}$ and $s\in\mathcal{S}'$ is a vertex color. Then, since $x=x'$ and $s\in\mathcal{S}'$, we must have $x=u=x'=u'$. Then $v\neq v'$ since $(u,v)\neq (u',v')$ by assumption. But this gives us that $(y,v,s_{2}),(y,v',s_{2})\in\mathcal{E}_{2}$, contradicting that $\mathcal{E}_{2}$ is an edge coloring. If $x\neq u$ then either $u\neq u'$ or $v\neq v'$ (or both). The former implies (by definition of $\mathcal{E}^{\circ}$ and $\mathcal{E}_{\circ}$) that $(\{x,u\},s''),(\{x,u'\},s'')\in\mathcal{E}_{1}$ for some $s''\in\mathcal{S}^{\circ}$, and the latter implies that $(y,v,s_{2}),(y,v',s_{2})\in\mathcal{S}_{2}$. In either way we get a contradiction as both $\mathcal{E}_{1}$ and $\mathcal{E}_{2}$ are edge colorings. If we assume that $(u,v)=(u',v')$ a similar contradiction can be reached.
\par
We now show that $\mathcal{E}$ is a strong edge coloring. Let \begin{equation}\label{eq5.1}((x,y),(u,v),(s,s_{2})),((x',y'),(u',v'),(s,s_{2}))\in\mathcal{E}.
\end{equation} Thus $(x',y')\neq(x,y)$ and $(u,v)\neq(u',v')$. We show there are no members $((x,y),(u',v'),(s',s'_{2}))$ or $((x',y'),(u,v),(s',s'_{2}))$ in $\mathcal{E}$ for any $(s',s'_{2})\in\mathcal{S}$.
\par
Suppose that $((x,y),(u',v'),(s',s'_{2}))\in\mathcal{E}$ for some $(s',s'_{2})\in\mathcal{S}$. Note that $(s',s'_{2})\neq(s,s_{2})$ as we have shown that $\mathcal{E}$ is an edge coloring. By definition of $\mathcal{E}$ we must have that $(y,v',s'_{2})\in\mathcal{E}_{2}$; and either $x=u'$ with $(x,s')\in\mathcal{V}$, or $(\langle x,u'\rangle,s')\in\mathcal{E}^{\circ}\cup\mathcal{E}_{\circ}$. Notice that if $y=y'$ and $s'_{2}\neq s_{2}$ then by (\ref{eq5.1}) we have $(y,v,s_{2}),(y,v',s_{2})\in\mathcal{E}_{2}$. This forces $v=v'$ since $\mathcal{E}_{2}$ is an edge coloring. But then we get that $(y,v,s_{2}),(y,v,s'_{2})\in\mathcal{E}_{2}$, a contradiction. Also notice that if $y\neq y'$ and $s'_{2}\neq s_{2}$ then by (\ref{eq5.1}) we have $(y,v,s_{2}),(y',v',s_{2}),(y,v',s'_{2})\in\mathcal{E}_{2}$, contradicting that $\mathcal{E}_{2}$ is a strong edge coloring. Thus, we may assume that $y=y'$ and $s'_{2}=s_{2}$, in which case, by (\ref{eq5.1}), we have $x'\neq x$, $s'\neq s$, and $v=v'$. We need to consider the two cases depending on whether $x=u'$ with $(x,s')\in\mathcal{V}$, or $(\langle x,u'\rangle,s')\in\mathcal{E}^{\circ}\cup\mathcal{E}_{\circ}$.
\newline
{\bf Case 1}: Assume that $x=u'$ with $(x,s')\in\mathcal{V}$. If $x=u$ and $(x,s)\in\mathcal{V}$ then we have $(x,s),(x,s')\in\mathcal{V}$, forcing $s=s'$, a contradiction.
\par
If $(\langle x,u\rangle,s)\in\mathcal{E}^{\circ}\cup\mathcal{E}_{\circ}$ and $s$ is an edge color then, by (\ref{eq5.1}), we must also have that $(\langle x',u'\rangle,s)\in\mathcal{E}^{\circ}\cup\mathcal{E}_{\circ}$. This gives us that $(\{x,u\},s),(\{x',u'\},s)\in\mathcal{E}_{1}$ which, since $x=u'$, and $\mathcal{E}_{1}$ is an edge coloring, we must also have $u=x'$. But this gives us that $(\langle x,u\rangle,s),(\langle x',u'\rangle,s)(=(\langle u,x\rangle,s))\in\mathcal{E}^{\circ}\cup\mathcal{E}_{\circ}$, which contradicts that the opposing orientations $\mathcal{E}^{\circ}$ and $\mathcal{E}_{\circ}$ have mutually disjoint sets of colors.
\newline
{\bf Case 2}: Now assume that $(\langle x,u'\rangle,s')\in\mathcal{E}^{\circ}\cup\mathcal{E}_{\circ}$. If $x=u$ and $(x,s)\in\mathcal{V}$ then $s$ is a vertex color so that, by (\ref{eq5.1}), we also have that $x'=u'$ and $(x',s)\in\mathcal{V}$. But this gives us that $(x,s),(u',s)(=(x',s))\in\mathcal{V}$ where $\{x,u'\}\in E_{1}$, which contradicts that $\mathcal{V}$ is a vertex coloring.
\par
If $(\langle x,u\rangle,s)\in\mathcal{E}^{\circ}\cup\mathcal{E}_{\circ}$ and $s$ is an edge color then, again by (\ref{eq5.1}), we must also have that $(\langle x',u'\rangle,s)\in\mathcal{E}^{\circ}\cup\mathcal{E}_{\circ}$. If $x=u'$ and $x'=u$ then we have that $(\langle x,u\rangle,s),(\langle x',u'\rangle,s)(=(\langle u,x\rangle,s))\in\mathcal{E}^{\circ}\cup\mathcal{E}_{\circ}$, which contradicts that the opposing orientations $\mathcal{E}^{\circ}$ and $\mathcal{E}_{\circ}$ have mutually disjoint sets of colors. If $\{x,u\}$ and $\{x',u'\}$ meet in exactly one point then the fact that $(\{x,u\},s),(\{x',u'\}，s)\in\mathcal{E}_{1}$ contradicts that $\mathcal{E}_{1}$ is an edge coloring. Then we must have that $x\neq u'$ and $x'\neq u$. But gives us that $(x,u,s),(x',u',s),(x,u',s')\in\mathcal{E}_{1}$, which contradicts that $\mathcal{E}_{1}$ is a strong edge coloring.
\par
If we assume that $((x',y'),(u,v),(s',s'_{2}))\in\mathcal{E}$ for some $(s',s'_{2})\in\mathcal{S}$, a similar contradiction can be reached. This completes the proof.
\end{proof}
\begin{remark} Note that the resulting strong chromatic index $sq(\Gamma_{1}:\Gamma_{2})=|(\mathcal{S}'\cup\mathcal{S}^{\circ}\cup\mathcal{S}_{\circ})\times \mathcal{S}_{2}|$ of $\Gamma_{1}:\Gamma_{2}$ reached in Lemma \ref{th5.1} exceeds that given in Theorem 5 of \cite{TOGN} for the standard strong product of graphs.
\end{remark}\noindent
{\it Proof of Theorem \ref{th5.2}}. Taking $\Gamma_{2}=(V_{2},W_{2},E_{2})$ and $\mathcal{E}_{2}$ to be the underlying bipartite graph and corresponding strong edge coloring of $\mathcal{P}'$, respectively, we can take $\Gamma_{1}$ to be $C_{m}$, i.e., the cycle on $m$ vertices. Notice $\Gamma_{1}$ has a vertex coloring $\mathcal{V}$ with colors from some $\mathcal{S}'$ where $|\mathcal{S}'|=2$, and a strong edge coloring $\mathcal{E}_{1}$ with colors from some $\mathcal{S}_{1}$ where $|\mathcal{S}_{1}|=3$. If $\mathcal{E}^{\circ}$ and $\mathcal{E}_{\circ}$ are two opposing orientations of $\mathcal{E}_{1}$ with mutually disjoint sets of colors $\mathcal{S}^{\circ}$ and $\mathcal{S}_{\circ}$, respectively (so that $|\mathcal{S}^{\circ}\cup\mathcal{S}_{\circ}|=6$), then by Lemma \ref{th5.1} we have that \[\mathcal{E}=\{((x,y),(u,v),(s,s_{2}))\mid(y,v,s_{2})\in\mathcal{E}_{2};\text{ and }x=u\text{ with }(x,s)\in\mathcal{V},\text{ or }(\langle x,u\rangle,s)\in\mathcal{E}^{\circ}\cup\mathcal{E}_{\circ}\}\] is a strong edge coloring on $\Gamma=\Gamma_{1}:\Gamma_{2}$ with colors from $\mathcal{S}=(\mathcal{S}'\cup\mathcal{S}^{\circ}\cup\mathcal{S}_{\circ})\times \mathcal{C}$. The result then follows from Lemma \ref{le1}.
\qed\par
Here we give an example.
\begin{example}\label{ex5.0} Let $a,b$ and $n$ be positive integers such that $a+b\leq n$, let $t=0$, and define $\Gamma_{2}=(V_{2},W_{2},E_{2})$, $\mathcal{E}_{2}$ and $\mathcal{C}$ as in Example \ref{ex3.0}. Let $m$ be a positive integer divisible by $6$, and let $\Gamma_{1}=C_{m}$ be the cycle on $m$ vertices. By Lemma \ref{th5.1} we have that \[\mathcal{E}=\{((y,A),(v,B),(s,C))\mid(A,B,C)\in\mathcal{E}_{1};\text{ and }y=v\text{ with }(y,s)\in\mathcal{S}',\text{ or }(\langle y,v\rangle,s)\in\mathcal{E}^{\circ}\cup\mathcal{E}_{\circ}\}\] is a strong edge coloring on $\Gamma=\Gamma_{1}:\Gamma_{2}$ with colors from $\mathcal{S}=(\mathcal{S}'\cup\mathcal{S}^{\circ}\cup\mathcal{S}_{\circ})\times \mathcal{C}$.
\par
The PDA constructed from $\Gamma$ in Theorem \ref{th5.2} will then have the following parameters.
\begin{center}
\tabcolsep=1.0cm
\small
\begin{tabular}{c|c|c|c}
\hline
  \hline
$K$&$1-M/N$&$F$&$R$\\\hline
$m\binom{n}{b}$&$(3/m)\binom{n-b}{a}/\binom{n}{a}$&$m\binom{n}{a}$&$(8/m)\binom{n}{a+b}/\binom{n}{a}$\\\hline
\end{tabular}
\end{center} This suggests that, for $m>6$, increasing the number of users by a factor of $m$ and decreasing the rate by a factor of $8/m$ can always be achieved by increasing the subpacketization level by a factor of $m$ and moderatley increasing the cache size.
\par
If we take $b=2$ then we obtain an $(m\binom{n}{2},m\binom{n}{a},m\binom{n}{a}-3\binom{n-2}{a},8\binom{n}{a+2})$-PDA. If we take $n=\lambda a$ for some constant $\lambda>1$, then we get $R=S/F=(8/m)\binom{n}{a+2}/\binom{n}{a}=(8/m)(\lambda-1)^{2}$ and $1-M/N=1-Z/F=(3/m)\binom{n-2}{a}/\binom{n}{a}\approx (3/m)((\lambda-1)/\lambda)^{2}$ (for $n$ large). By Stirling's formula we have that (for $n$ large) $K=m\binom{n}{2}\approx (m/2)n^{2}$, and \[
F=\binom{n}{\lambda^{-1} n}\approx \frac{1+o(1)}{\sqrt{2\pi\lambda^{-1}(1-\lambda^{-1})n}}2^{nH(\lambda^{-1})}=\mathcal{O}(K^{-1/4}2^{\sqrt{(2/m)K}H(\lambda^{-1})}), 
\]where $H$ is defined as in Equation (\ref{eqa}). The rate at which $F$ grows with $K$ is sub-exponential, but still is an improvement from that of the initial (component) PDA, whose growth rate can be found in Remark IV.3 of \cite{CHONG}.
\end{example}
\begin{example}\label{ex5.1} In \cite{TANG} it was shown that for any two positive integers $n>l$, taking $q$ a prime power, and $x$ the least positive integer such that $(l+1) | nx$, there exists a $(K,F,Z,S)$ PDA where $F=(q-1)q^{l}xn/(l+1),Z=F-g$ and $S=xq^{l}$ (where $g=x(q-1)q^{l-1}$). Such a PDA yields a caching scheme with $M/N=1-(l+1)/(nq)$ and $R=(l+1)/((q-1)n)$. Then applying Theorem \ref{th5.1} to such a PDA yeilds an $(mK,mF,mF-3g,8S)$ PDA whence a caching scheme with $M/N=1-(l+1)/(mnq)$ and $R=8(l+1)/(mn(q-1))$.
\end{example}
\section{Comparisons between known PDAs and the PDAs resulting from applying our method to the known PDAs}\label{sec6}
In this section we will compare the parameters of some PDAs obtained using the methods discussed in this paper to those of the component PDAs (i.e., the PDAs combined to give the new PDA). \par Here we compare the scheme obtained in Example \ref{ex3.0} to its componenet schemes, i.e., to Scheme 1 of \cite{CHONG}, with a fixed cache ratio of $M/N=3/4$.
\\
\\
\\
\begin{minipage}[c]{0.45\textwidth}
\begin{center}
\tabcolsep=0.35cm
\small
\begin{tabular}{c|c|c|c|c}
\hline
  \hline
$n\approx$&$K$&$1-M/N$&$F$&$R$\\\hline
$6\sqrt{5}$&$90$&$1/4$&$2382$&$1$\\\hline
$2\sqrt{66}$&$132$&$1/4$&$15406$&$1$\\\hline
$2\sqrt{91}$&$182$&$1/4$&$101147$&$1$\\\hline
\end{tabular}
\captionof{table}{{\scriptsize Scheme 1 of \cite{CHONG} with $b=2,n=\lambda a$ and $\lambda=2$ (and using Stirling's formula to compute $F$).}}
\end{center}
\end{minipage}%
\hspace{0.05\textwidth}%
\begin{minipage}[c]{0.45\textwidth}
\begin{center}
\tabcolsep=0.35cm
\small
\begin{tabular}{c|c|c|c|c}
\hline
  \hline
$n$&$K$&$1-M/N$&$F$&$R$\\\hline
$10$&$90$&$1/4$&$210$&$5$\\\hline
$12$&$132$&$1/4$&$792$&$7$\\\hline
$14$&$182$&$1/4$&$3003$&$8$\\\hline
\end{tabular}
\captionof{table}{{\scriptsize Example \ref{ex3.0} with $a=t=1$ and $n=\lambda b$ with $\lambda=2$}}
\end{center}
\end{minipage}
\\
\\
Comparing Tables II and III we can see that, at a cost of moderately increasing the rate we get a substantial reduction in $F$.\par
Here we compare the scheme obtained in Example \ref{ex4.2} to its component schemes, i.e., to Theorem 3 of \cite{YAN1}, with a fixed cache ratio of $M/N=3/4$.
\\
\\
\\
\begin{minipage}[c]{0.45\textwidth}
\begin{center}
\small
\begin{tabular}{c|c|c|c|c}
\hline
  \hline
$n\approx$&$K$&$1-M/N$&$F$&$R$\\\hline
$\sqrt{1568}$&$784$&$1/4$&$2598778$&$39.50$\\\hline
$\sqrt{2592}$&$1296$&$1/4$&$255881905$&$50.91$\\\hline
$\sqrt{4050}$&$2025$&$1/4$&$45902134943$&$63.64$\\\hline
\end{tabular}
\captionof{table}{{\scriptsize Theorem 3 of \cite{YAN1} with $a=2,\lambda=1,n=\mu b$ and $\mu=2\sqrt{2}+4$ (using Stirling's formula to compute $F$).}}
\end{center}
\end{minipage}%
\hspace{0.05\textwidth}%
\begin{minipage}[c]{0.45\textwidth}
\begin{center}
\tabcolsep=0.35cm
\small
\begin{tabular}{c|c|c|c|c}
\hline
  \hline
$a$&$K$&$1-M/N$&$F$&$R$\\\hline
$4$&$784$&$1/4$&$5215$&$16$\\\hline
$4.5$&$1296$&$1/4$&$18542$&$20.25$\\\hline
$5$&$2025$&$1/4$&$66754$&$25$\\\hline
\end{tabular}
\captionof{table}{{\scriptsize Example \ref{ex4.2} with $a=a',b=b'=2,t=t'=1,n=n'=\lambda a$ with $\lambda=2$.}}
\end{center}
\end{minipage}
\\
\\

Comparing Tables IV and V we can (for these particular values of $K$ and $M/N$) see a significant reduction in both $F$ and $R$. 
\par
Here we compare the scheme obtained in Example \ref{ex5.0} to a component scheme, i.e., again to Scheme 1 of \cite{CHONG}, with a fixed cache ratio of $M/N=3/4$.
\\
\\
\\
\begin{minipage}[c]{0.45\textwidth}
\begin{center}
\small
\begin{tabular}{c|c|c|c|c}
\hline
  \hline
$n\approx$&$K$&$1-M/N$&$F$&$R$\\\hline
$\sqrt{540}$&$270$&$1/4$&$1637369$&$1$\\\hline
$\sqrt{792}$&$396$&$1/4$&$44564986$&$1$\\\hline
$\sqrt{1092}$&$546$&$1/4$&$1230404836$&$1$\\\hline
\end{tabular}
\captionof{table}{{\scriptsize Scheme 1 of \cite{CHONG} with $b=2,n=\lambda b$ and $\lambda=2$ (using Stirling's formula to compute $F$).}}
\end{center}
\end{minipage}%
\hspace{0.05\textwidth}%
\begin{minipage}[c]{0.45\textwidth}
\begin{center}
\tabcolsep=0.35cm
\small
\begin{tabular}{c|c|c|c|c}
\hline
  \hline
$n$&$K$&$1-M/N$&$F$&$R$\\\hline
$10$&$270$&$1/4$&$703$&$7.77$\\\hline
$12$&$396$&$1/4$&$2152$&$7.77$\\\hline
$14$&$546$&$1/4$&$6679$&$7.77$\\\hline
\end{tabular}
\captionof{table}{{\scriptsize Example \ref{ex5.0} with $m=6,b=2$ and $n=\lambda a$ with $\lambda=2$.}}
\end{center}
\end{minipage}
\\
\\

Comparing Tables VI and VII we can see that at a cost of increasing the rate by a multiple of $7.77$ we get a significant reduction in $F$.
\par
Here we compare the scheme obtained in Example \ref{ex5.1} to a component scheme, i.e., to a scheme obtained in \cite{TANG} (see Example \ref{ex5.1} of this paper), with a fixed cache ratio of $M/N=5/8$.
\\
\\
\\
\begin{minipage}[c]{0.45\textwidth}
\begin{center}
\small
\begin{tabular}{c|c|c|c|c}
\hline
  \hline
$(n,q,l,x)$&$K$&$1-M/N$&$F$&$R$\\\hline
$(24,2,17,3)$&$48$&$3/8$&$2^{19}$&$3/4$\\\hline
$(60,2,44,4)$&$120$&$3/8$&$2^{46}$&$3/4$\\\hline
$(96,2,71,3)$&$192$&$3/8$&$2^{73}$&$3/4$\\\hline
\end{tabular}
\captionof{table}{{\scriptsize Scheme obtained in \cite{TANG} (see Example \ref{ex5.1} of this paper).}}
\end{center}
\end{minipage}%
\hspace{0.05\textwidth}%
\begin{minipage}[c]{0.45\textwidth}
\begin{center}
\tabcolsep=0.35cm
\small
\begin{tabular}{c|c|c|c|c}
\hline
  \hline
$(n,q,l,x)$&$K$&$1-M/N$&$F$&$R$\\\hline
$(4,2,5,3)$&$48$&$3/8$&$3\cdot 2^{7}$&$2$\\\hline
$(10,2,14,3)$&$120$&$3/8$&$3\cdot 2^{16}$&$2$\\\hline
$(16,2,23,3)$&$192$&$3/8$&$3\cdot 2^{25}$&$2$\\\hline
\end{tabular}
\captionof{table}{{\scriptsize Example \ref{ex5.0} with $m=6,b=2$ and $n=\lambda a$ with $\lambda=2$.}}
\end{center}
\end{minipage}
\\
\\
Comparing Tables VIII and IX we can see that at a cost of increasing the rate by a multiple of $3/2$ we get a significant reduction in $F$.

\section{Conclusion}\label{sec7}
We have considered various methods of combining one or more strong edge colorings for the purpose of obtaining new strong edge colorings of bipartite graphs and therefore, new PDAs. The methods we have considered mostly involve various types of products of graphs. We have also investigated how the parameters of certain strong edge colorings (hence certain PDAs) change after being combined with others, and have compared the parameters of the resulting PDAs with those of known ones.


%





\ifCLASSOPTIONcaptionsoff
  \newpage
\fi

\end{document}